\documentclass[twoside]{amsart}
\usepackage{amscd}
\usepackage[all]{xy}

\usepackage{graphicx}

\usepackage{xcolor}

\newcommand{\Aut}{{\operatorname{Aut}}}

\newcommand{\sF}{\mathcal{F}}

\newcommand{\PP}{\mathbb{P}}

 \newcommand{\sC}{\mathcal{C}}

\newcommand{\sO}{\mathcal{O}}

\newcommand{\sB}{\mathcal{B}}

 \newcommand{\sH}{\mathcal{H}}

\newcommand{\sX}{\mathcal{X}}
 
\newcommand{\sI}{\mathcal{I}}

 \numberwithin{equation}{section}

\theoremstyle{plain}
\newtheorem{thm}[equation]{Theorem}

\newtheorem{prop}[equation]{Proposition}
\newtheorem{lm}[equation]{Lemma}

\theoremstyle{definition}
\newtheorem{defn}[equation]{Definition}

\newtheorem{rk}[equation]{Remark}

\begin{document}

\title{Unirationality of certain universal families of cubic fourfolds.}

\author[H. Awada]{Hanine Awada}
\address{Institut Montpellierain Alexander Grothendieck \\ %
Universit\'e de Montpellier \\ CNRS %
Case Courrier 051 - Place Eug\`ene Bataillon \\ %
34095 Montpellier Cedex 5 \\ %
France}
\email{hanine.awada@umontpellier.fr}

\author[M. Bolognesi]{Michele Bolognesi}
\address{Institut Montpellierain Alexander Grothendieck \\ %
Universit\'e de Montpellier \\ CNRS %
Case Courrier 051 - Place Eug\`ene Bataillon \\ %
34095 Montpellier Cedex 5 \\ %
France}
\email{michele.bolognesi@umontpellier.fr}

\begin{abstract} 
The aim of this short note is to define the \it universal cubic fourfold \rm over certain loci of their moduli space. Then, we propose two methods to prove that it is 
unirational over the Hassett divisors $\sC_d$, in the range $8\leq d \leq 42$. 
By applying inductively this argument, we are able to show that, in the same range of values, $\sC_{d,n}$ is unirational for all integer values of $n$.
Finally, we observe that for explicit infinitely many values of $d$, the universal cubic fourfold over $\sC_d$ can not be unirational.
\end{abstract}

\maketitle 
\section {Introduction}
Since the celebrated result of Clemens and Griffiths \cite{CG} on the irrationality of smooth cubic threefolds, at least three generations of algebraic geometers have been working hard on the problem of rationality of cubic fourfolds. The search for the good invariant that would detect (ir)rationality has involved Hodge theoretical \cite{Has1}, homological \cite{kuz4fold}, K-theoretic \cite{GS} and motivic \cite{BP} methods. While the general suspicion is that the generic cubic fourfold  should be non rational as well, no explicit example of irrational cubic fourfold is known at the moment. A little more is known, and the expectation are slightly more precise, about rationality. Let $\sC$ denote the moduli space of smooth cubic fourfolds and $\sC_d\subset \sC$ the Hassett divisor of discriminant $d$, then a well known conjecture due to Harris-Hassett-Kuznetsov (see \cite{kuz4fold,Has1,AT}) states that there should be a countable infinity of Hassett divisors that parametrize rational cubic fourfolds. In particular these divisors should correspond to cubic fourfolds that have an associated K3 surface, in a Hodge theoretical or derived categorical sense (see \cite{kuz4fold,Has00}). In the second decade of this century, the search for rational examples of cubic fourfolds has concentrated more on cubics containing some special surfaces \cite{ABBVA,Nu,dp6,hanine}. This has been developed further in the works of Russo-Staglian\'o, and at the moment the HHK conjecture is known to be true for four low values of $d$ \cite{BD,BRS,RS1,RS2}. In all these cases, the special cubic fourfolds involved contain rational surfaces.

\medskip

In this note we are also interested in the birational geometry of cubic fourfolds, but under a slightly different point of view. Universal families are well known and studied objects in the realm of moduli spaces of curves, where forgetful maps are the algebraic geometer's everyday tools. Also in the context of K3 surfaces several results about these objects have appeared in the literature, also very recently \cite{FV1,FV2,Ma}. Probably because the birational geometry of a single cubic fourfold is already so difficult to understand, the problem of studying the same problems for universal families of cubics seems to have been overtaken. Thanks to the recent advances in the study of rationality of a single cubic fourfold, it seems however natural to ask what can be said in families, in particular over certain relevant divisors $\sC_d$ in the moduli space. In this paper we prove the following fact.

\begin{thm}\label{bigthm}
The universal cubic fourfolds $\sC_{d,1}$ are unirational in the range $8\leq d \leq 42$.
\end{thm}

More precisely, we propose two different methods, since each one seems to have its own interest, concerning different aspects of the geometry of cubic fourfolds. On one hand, we use some recent results of Farkas-Verra about universal K3 surfaces and use the relation with cubic fourfolds given by associated K3s.
On the other hand, we use the presentation of Hassett divisors as cubics containing certain rational surfaces, as done by Nuer and Russo-Staglian\'o, and Kollar's theorem on unirationality of cubics. The second argument, applied inductively, also allows us to prove the following fact.

\begin{thm}\label{nd}
The universal families $\sC_{d,n}$ are unirational for all values of $n$ if $8\leq d \leq 42$.
\end{thm}

Of course, though the intuitive definition is clear, this requires a few checks that actually the universal family exists, which we make in Sect. 2. In the same Section, we also collect some important properties of cubic fourfolds, the rational scrolls they contain, their associated K3 surfaces and their Fano varieties of lines, that are then necessary in Section 3. Section 3 is devoted to the proof of unirationality of $\sC_{26,1}$ and $\sC_{42,1}$ via universal K3s, which is based on the construction of a couple of rational/unirational moduli quotients involving the families of scrolls inside our cubic fourfolds. 

\smallskip

In Sect. 4 we develop the second argument, which is slightly more general, and allows us to complete the missing cases of Thm. \ref{bigthm}. We use thoroughly the recent theory of rational surfaces contained in special cubic fourfolds of low discriminants, and combine this with Kollar's Theorem on unirationality of cubic hypersurfaces.

\smallskip

In Section 5 we underline the special nature of the cases we are considering, by comparing them with preceding results of  Gritsenko-Hulek-Sankaran \cite{GHS13} and V\'arilly-Alvarado-Tanimoto \cite{TVA}.

In fact, in \cite{GHS13} it is proven that the moduli space of polarized K3 surfaces of degree $2m$ has positive Kodaira dimension or is even of general type for an infinity of values (see Prop. \ref{gtk3} for details).

Furthermore one expects that the gaps in Prop. \ref{gtk3} can be filled in by using automorphic form techniques. V\'arilly-Alvarado and Tanimoto \cite{TVA} have started pursuing this project obtaining another relevant range of values where the Kodaira dimension of $\sC_d$ is positive (see Prop. \ref{gtCd}).

These results give us \it negative \rm information about the divisors $\sC_d$ where the universal cubic $\sC_{d,1}$ can not be unirational.

\begin{prop}\label{notuni}
The universal cubic fourfold $\sC_{d,1}$ is not unirational if:

\begin{enumerate}
\item $d > 80,\ d \equiv 2\ (mod\ 6),\ 4\not \vert  d$ and such that for any odd prime $p,\ p | d$
implies $p \equiv 1\ (mod\ 3)$;


\item $d=6n+2$ for $n > 18$ and $n \neq 20, 21, 25$;

\item $d=6n+2$, $n>13$ and $n\neq 15$;

\item $d=6n$, and $n >18$, $n\neq 20, 22, 23, 25, 30, 32$;

\item $d=6n$ for $n>16$ and $n\neq 18, 20, 22, 30$.
\end{enumerate}

\end{prop}



We wish to thank Francesco Russo, Giovanni Staglian\'o and Sandro Verra for lots of discussions on related topics in the last few months. Thanks to Shouhei Ma and Zhiwei Zheng for giving us precise references and information.

\section{The existence of the universal cubic fourfold, and some properties of scrolls and associated K3 surfaces.}

As it is customary when discussing moduli spaces, one of the first questions one considers is the actual existence of a universal family. The invariant theory of cubic fourfolds, with respect to the natural $PGL(6)$-action, is known by the work of Laza \cite{lazaGIT,lazaperiod}. 

\begin{thm}\label{lazarone}\cite[Thm. 1.1]{lazaGIT}
A cubic fourfold $X\subset \PP^5$ is not GIT stable if and only if one of the following conditions holds:

\begin{enumerate}
\item $X$ is singular along a curve $C$ spanning a linear subspace of $\PP^5$ of dimension at most 3;
\item $X$ contains a singularity that deforms to a singularity of class $P_8,X_9$ or $J_{10}$.
\end{enumerate}

In particular, if $X$ is a cubic fourfold with isolated singularities, then X is stable if and only if X has at worst simple singularities.
\end{thm}
 
The GIT quotient that we are considering is slightly different than the one constructed in \cite{lazaGIT}, since we only consider smooth cubic fourfolds. That is, we consider the complement $U\subset |\sO_{\PP^5}(3)|$ of the discriminant and take the quotient $\sC:= U//PGL(6)$. The resulting quotient is a quasi-projective variety of dimension 20 and Thm. \ref{lazarone} assures that all point are stable in the GIT sense. We are in particular concerned by the divisors  $\sC_d$ inside $\sC$, with $8\leq d \leq 42$. By the seminal work of Hassett \cite{Has00}, and further developments \cite{Nu,RS1,RS2,FV1,FV2}, they can be characterized in terms of \it special \rm surfaces - not homologous to linear sections - contained in the generic cubic fourfold in each divisor. Recall in fact that for the general cubic fourfold $H^{2,2}(X,\mathbb{Z})=\mathbb{Z}$, generated by a linear section, and the locus where the rank of this group is bigger than one is the union of a countable infinity of divisors in $\sC$. For example, we will need the following descriptions.

\begin{defn}\label{defscrolls}
We have:
\begin{itemize}
\item $\sC_{26}:=\overline{\{\mathrm{Cubic\ fourfolds\ containing\ a\ 3-nodal\ septic\ scroll}\}}$
\item $\sC_{42}:=\overline{\{\mathrm{Cubic\ fourfolds\ containing\ a\ 8-nodal\ degree\ 9\ scroll}\}}.$
\end{itemize}
\end{defn}

Moreover, cubic fourfolds $X$ inside these two divisors have an associated K3 surface\footnote{Actually cubics in $\sC_{42}$ have two associated K3s, we will see more details about this later}. As it is well known, this means that the \it Kuznetsov component \rm of the bounded derived category of $X$ is equivalent to the derived category of a K3 surface, or equivalently that the Hodge structure of the nonspecial cohomology of the cubic fourfold is essentially the Hodge structure of the primitive cohomology of a K3 surface. We refrain to give more details about this, since these notions are now well known and the references in the literature are very good \cite{kuz4fold,Has00,Has1}. The associated K3 surfaces for cubics from $\sC_{26}$ have genus 14 and degree 26, whereas for cubics from $\sC_{42}$ they have genus 22 and degree 42. 

\medskip


More generally, thanks in particular to the work of Nuer \cite{Nu} (see also \cite[Sect. 4]{RS1}), it is known that  if $d\leq 42$, then the cubics in each $\sC_d$ can be defined as the cubics containing certain rational surfaces. These rational surfaces, that are introduced in \cite[Sect. 3]{Nu}, are obtained as images of $\PP^2$ via linear systems with prescribed fixed locus. For each value of $d$ we have a different linear system. We will see more about this in Sect. \ref{unirat}.





\medskip

Let $\sF_g$ denote the moduli space of genus $g$ K3 surfaces. For $X\in \sC_d$ with discriminant $d = 2(n^2 + n + 1)$,
with $n \geq 2$ (remark that 26 and 42 all verify this equality), Hasset shows \cite[Sect. 6]{Has00} that
there is an isomorphism

\begin{equation}\label{isofano}
F(X) \cong S^{[2]}
\end{equation}

between the Fano variety of lines $F(X)$ and the Hilbert scheme $S^{[2]}$ of couple of points of $(S,H)$ a polarized K3 surface with $H^2 = d$, the K3 \it associated \rm to $X$. The moduli spaces of the corrisponding K3 surfaces have dimension 19, the same dimension as divisors in $\sC$, and this assignment induces a rational map

\begin{equation}\label{birat}
\sF_{\frac{d+2}{2}} \dashrightarrow \sC_d
\end{equation}

that is birational for $d \equiv 2\ (mod\  6)$ and degree two for $d \equiv 0 (mod\ 6)$. In particular, for $\sC_{26}$ and $\sC_{42}$, we have rational maps

\begin{eqnarray*}
\sF_{14} & \stackrel{\sim}{\dashrightarrow} & \sC_{26}; \\
\sF_{22} & \stackrel{2:1}{\dashrightarrow} & \sC_{42}. 
\end{eqnarray*}

Another upshot of the isomorphism (\ref{isofano}) is that, once we fix a smooth cubic fourfold in $\sC_{26}$ or $\sC_{42}$, the family of scrolls contained inside $X$ is precisely parametrized by the associated K3. The construction, roughly speaking, goes as follows. 
For each $p\in S$, one defines a rational curve 

$$\Delta_p :=\{ y\in S^{[2]}:\{p\} = Supp(y)\}.$$

The image of $\Delta_p$ inside $F(X)$ then defines a (possibly singular) scroll $R_p \subset X$.

\begin{prop}\label{scrollk3}\cite{FV1,FV2}
\begin{enumerate}
\item The family of septimic 3-nodal scrolls inside a generic cubic fourfold $X$ in $\sC_{26}$ is the genus 14 K3 surface associated to $X$.
\item The family of 8-nodal degree 9 scrolls inside a generic cubic fourfold $X$ in $\sC_{42}$ is the genus 22 K3 surface associated to $X$.
\end{enumerate}
\end{prop}


Recall that the generic K3 surface in $\sF_{14}$ and $\sF_{22}$ has no automorphism, hence universal families exist:

\begin{eqnarray*}
f_{14}: \sF_{14,1} & \to & \sF_{14}; \\
f_{22}: \sF_{22,1} & \to & \sF_{22}.
\end{eqnarray*}

at least over an open subset of each moduli space. This has allowed the study of $\sF_{14,1}$ and $\sF_{22,1}$ by Faraks-Verra. They considered universal K3s as moduli spaces for couples $(X,R)$, where $X$ is a cubic fourfold, and $R$ is a scroll - belonging to a given class of surfaces - contained in $X$.  

\begin{prop}\cite{FV1,FV2}\label{uniratK3s}
The universal K3s $\sF_{8,1}$ and $\sF_{22,1}$ are unirational.
\end{prop}

Let us now come to the definition of the object that we will study in Sect. 3. Similarly to the case of curves and K3 surfaces we give the following definition.

\begin{defn}
By \it universal cubic fourfold \rm over a divisor $\sC_d$ we mean the moduli space $\sC_{d,1}$ of 1-pointed cubic fourfolds.
\end{defn}

As it is customary over moduli spaces, the mere existence of a universal cubic fourfold needs a little bit of justification. 

\begin{prop}\label{rigidity}
The generic cubic fourfold in any divisor $\sC_d$ does not have projective automorphisms, hence a universal family of cubic fourfolds exists over an open subset of each divisor.
\end{prop}
 
\begin{proof}
This appears to be an easy corollary of \cite[Thm 3.8]{autom} where it is proven that the families of cubic fourfolds with a non-trivial projective automorphism have dimension at most 14. Hence clearly also the generic element of any divisor has no automorphism
\end{proof}

\begin{rk}
Given the result of \cite{autom}, Prop. \ref{rigidity} is straightforward and very general. For $d=26,$ and $42$, since the isomorphism (\ref{isofano}) holds, one could also argue in a more ``geometric'' way as follows. Any automorphism $\phi \in \Aut (X)$ of a smooth cubic fourfold $X$ induces an automorphism of the Fano variety of lines $F(X)$. Then, via the isomorphism $F(X)\cong S^{[2]}$ with the symmetric square of the associated K3, one can use the fact that the generic K3 has no non-trivial automorphism to conclude that $\phi$ is the identity. Remark also that, by Prop. \ref{scrollk3}, the fact that the associated K3 has no non-trivial automorphism implies that any $\phi \in \Aut (X)$ should send each scroll onto itself.
\end{rk}


\section{Unirationality for $\sC_{26,1}$ and $\sC_{42,1}$ via universal K3 surfaces.}\label{fv}

By the results of the preceding section, at least over a dense subset of $\sC_{26}$ and $\sC_{42}$, there exists a universal family of cubic fourfolds. For simplicity, and since we are however working in the birational category, we will still denote the two universal families by $\sC_{26,1} \to \sC_{26}$ and $\sC_{42,1}\to \sC_{42}$, without stressing the fact that they may not be defined everywhere. 

\medskip

In this section we will show that:

\begin{thm}
The universal cubic fourfolds $\sC_{26,1}$ and $\sC_{42,1}$ are unirational.
\end{thm}

The strategy will be the same for the two cases, hence we will resume here below shortly the properties, that hold for both divisors, that we will need. In order to keep the notation not too tedious, when we will say that a certain property holds for $\sC_n$ we will assume $n=26,42$.\\
 
Let $X$ be a generic smooth cubic fourfold in $\sC_n$. We will denote by $K(X)$ the associated K3 surface and by $S(X)$ the family of scrolls (as defined in Def. \ref{defscrolls}) contained in $X \subset\PP^5$.  For $n=26$ septimic 3-nodal scrolls, for $n=42$ 8-nodal degree 9 scrolls. One can rephrase Prop. \ref{scrollk3} by saying that $K(X)$ is isomorphic to $S(X)$. More generally, we will denote by $\mathcal{S}_n$ the Hilbert scheme of scrolls contained in cubics in $\sC_n$ (those appearing in Def. \ref{defscrolls}), and by $S_n$ the 
$PGL(6)$-quotient of $\mathcal{S}_n$. These moduli quotients have been also considered in \cite{FV1,FV2,Lai}. Taking example from these papers let us give the following definition.



\begin{defn}
Let us denote by

$$\mathcal{H}_n= \{
(X, R) : R \subset X,\ [X] \in \sC_n,\ R \in S(X) \}$$

the ''nested" Hilbert scheme given by the couples ($X$ cubic fourfold whose class lives in $\sC_n$) + ($R$ rational normal scroll in $S(X)$); and by $H_n:= \mathcal{H}_n // PGL(6)$ the corresponding moduli quotient.
\end{defn}



From \cite{FV1,FV2} we have the following result.

\begin{prop}\label{biratXi}
The universal K3 surfaces $\sF_{n,1}$ are birational to $H_n$.
\end{prop}

In particular, in \cite{FV1,FV2} the authors show that $\mathcal{S}_n$ is unirational and $\mathcal{H}_n$ is birational to a projective bundle over $\mathcal{S}_n$, and hence unirational as well. We add just one more character to the plot, that is:

\begin{defn}
We denote by $\mathcal{H}_{n,1}$ the Hilbert scheme of the triples $(X, R, p)$ as follows:

$$\mathcal{H}_{n,1}= \{
(X, R, p)) : R \subset X,\ [X] \in \sC_n,\ R \in S(X),\ p \in X \}.$$

We will denote by $H_{n,1}$ the corresponding moduli  quotient by $PGL(6)$.
\end{defn}

Of course, as a consequence of Prop. \ref{rigidity}, this is generically a fibration in cubic hypersurfaces over $\mathcal{H}_n$. Finally we will denote by $\mathcal{H}_n^3$ the open subset complementary to the discriminant inside $|\mathcal{O}_{\PP^5}(3)|$\footnote{the was previously denoted by $U$ - but it seems more coherent to use this notation here}, and by $\mathcal{H}_ {n,1}^3$ the universal family over $\mathcal{H}_n^3$. 

 
\begin{thm}\label{ratxi}
The Hilbert scheme $\mathcal{H}_{n,1}$ is unirational.
\end{thm}

\begin{proof}

Let us try to give a global picture of the situation.

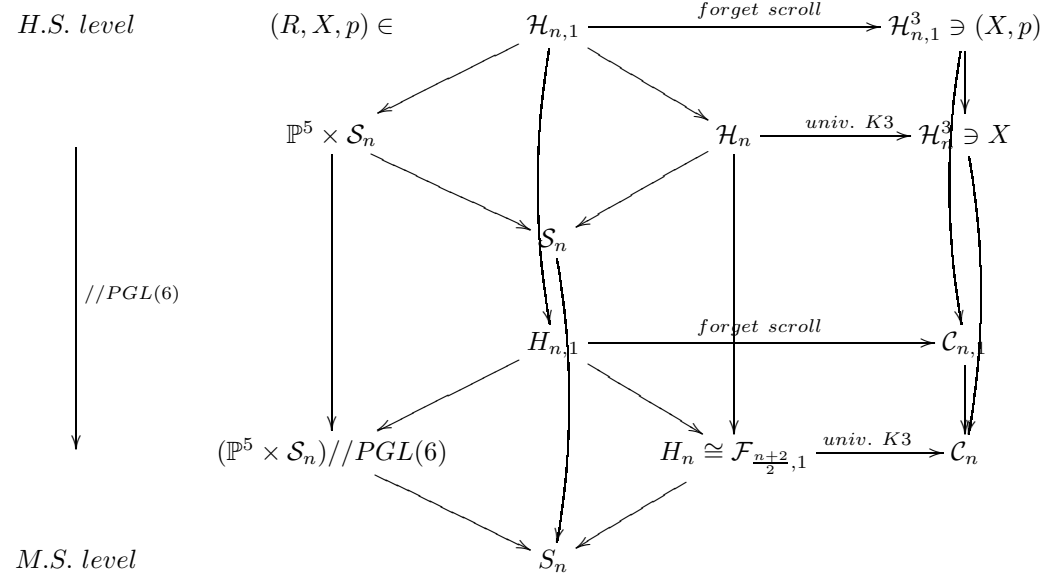
\begin{equation}\label{diagrammone}
\xymatrix{
H.S.\ level& (R,X,p) \in & \mathcal{H}_{n,1}\ar[dr]\ar[dl] \ar@/_/[ddd]\ar[rr]^{forget\ scroll} &  &
\mathcal{H}^3_{n,1} \ar@/_/[ddd]\ar[d]  \ni (X,p) \\
{} \ar[ddd]^{//PGL(6)} & \PP^5 \times \mathcal{S}_n \ar[ddd]\ar[dr] & & \mathcal{H}_n \ar[dl]\ar[ddd]\ar[r]^{univ.\ K3} & \mathcal{H}_n^3\ar@/^/[ddd]  \ni X  \\
& & \mathcal{S}_n\ar@/^/[ddd] & & \\
& & H_{n,1}\ar[rr]^{forget\ scroll}\ar[dl]\ar[dr] & & \sC_{n,1}\ar[d] \\
& (\PP^5 \times \mathcal{S}_n)//PGL(6) \ar[dr]& & H_n \cong 
\mathcal{F}_{\frac{n+2}{2},1} \ar[r]^{\ \ \ \ univ.\ K3}\ar[dl] & \sC_n \\
M.S.\ level& & S_n & &  }
\end{equation}

Here above all vertical arrows shall be intended as quotients by the $PGL(6)$-action. On the \it upper \rm level of the diagram all spaces are Hilbert schemes, whereas on the \it lower \rm they are moduli quotients (H.S  and M.S. for short). By the unirationality of $\mathcal{S}_n$ we have the unirationality of $\PP^5 \times \mathcal{S}_n$. Then we observe that there is a natural forgetful map 

\begin{eqnarray}
\phi_n: \mathcal{H}_{n,1} & \to & \PP^5 \times \mathcal{S}_n;\\
(X,R,p) & \mapsto & (p,R).
\end{eqnarray}

Namely, the fiber over a given couple $(R,p)$ is exactly the linear system \\ $\PP H^0(\PP^5,\sI_{R\cup p}(3))$ of cubics in $\PP^5$ containing $R$ and $p$.
Over an open subset of $\PP^5 \times \mathcal{S}_n$ (i.e. the locus where $p \not\in R$) the rank does not change and it is 11, and 5 respectively for $d=26$ and $42$. This in turn implies that $\mathcal{H}_{n,1}$ is unirational.
\end{proof} 

\begin{prop}
The universal cubic fourfold $\sC_{n,1}\to \sC_n$ is unirational.
\end{prop}

\begin{proof}
Recall from Prop. \ref{rigidity} that the generic cubic fourfold in the divisors $\sC_n$ has no projective automorphism. Call $V\subset \sC_n$ the dense locus where cubics have no automorphism. It is straightforward to see - and we have already implicitely used this in Diag. \ref{diagrammone} - that, over $V$, the universal cubic has a natural quotient structure

\begin{equation}
\sC_{n,1} = \mathcal{H}_{n,1}^3// PGL(6).
\end{equation}

Hence $\sC_{n,1}$ is the natural moduli space for couples $(X,p)$, up to the action of $PGL(6)$. The upshot is that there exist a natural rational surjective forgetful map (up to $PGL(6)$-action)

\begin{eqnarray}
\varphi_R: H_{n,1} & \to & \sC_{n,1}; \\
(X,R,p) & \mapsto & (X,p);
\end{eqnarray}

that forgets the scrolls contained in $X$. By Thm. \ref{ratxi}, the variety $H_{n,1}$ is unirational, and it dominates $\sC_{n,1}$, thus also $\sC_{n,1}$ is unirational.
\end{proof}

\section{Unirationality through rational special surfaces.}\label{unirat}

\subsection{Special cubics in $\sC_d$ in the range $8\leq d \leq 38$.}\label{sectnuer}

The goal of this Section is to prove the unirationality of the non-empty families of universal cubics $\sC_{d,1}$ for $8\leq d\leq 38$ by generalizing some results of Nuer, and using a celebrated result of Kollar \cite{Ko}. Then, by applying inductively the same argument, we will show that $\sC_{d,n}$ - the universal cubic with $n$ marked points - is also unirational in the same range. Let us recall shortly the results we need.

\medskip

In his paper \cite{Ko}, extending a result of Segre that held only over $\mathbb{Q}$, Kollar shows the following

\begin{thm}
Let $k$ be a field and $X \subset \PP^{
n+1}$ a smooth cubic hypersurface of dimension $n \geq 2$ over $k$. Then the following are equivalent:

\begin{enumerate}
\item $X$ is unirational (over $k$);
\item $X$ has a $k$-point.
\end{enumerate}
\end{thm}

We are going to use this in a relative setting. In order to do this we recall some results due to Nuer.

\medskip

In \cite{Nu}, the author studies the birational geometry of the divisors $\sC_d$, for $12 \leq d \leq 38$, by considering an open subset $U_d \subset (\PP^2)^p$ parametrizing generic $p$-tuples of distinct points that give certain cohomological
invariants that appear in \cite[Table 2]{Nu}. Let $S$ be the rational surface obtained as the blow-up of $\PP^2$ along the $p$ points $x_1, \dots, x_p$, and let us embed it into $\PP^5$ via the linear series $|aL - (E_1 +\dots + E_i) - 2(E_{i+1}+ \dots + E_{i+j}) - 3(E_{i+j+1} +\dots +E_p)|$, where the parameters $(a,i,j,p)$ of the linear system are those displayed in \cite[Table 1]{Nu}. From \cite[Thm. 3.1]{Nu} we know that there exists a vector bundle $V_d \to U_d$ such that the fiber over $(x_1, \dots , x_p)$ is the space of global sections $H^0(\mathcal{I}_{S|P5}(3))$. 
The natural classifying morphism $\PP(V_d) \to \sC_d$, obtained by the universal property of $\sC_d$ is dominant and this implies the unirationality of $\sC_d$.

\medskip

Let us now denote by $\sX_d$ the universal cubic over $\PP(V_d)$, and by $\pi_{p+1}: (\PP^2)^{p+1} \to (\PP^2)^p$ the forgetful map  that forgets the last point. Now we can claim :

\begin{thm}\label{teokol}
The universal cubic fourfolds $\sC_{d,1}$ are unirational if $12\leq d \leq 38$.
\end{thm}

\begin{proof}
By pulling back through $\pi$, we have the following commutative diagram. 

\begin{eqnarray}
\xymatrix{
\pi^* \sX_d \ar[r]\ar[d] & \sX_d \ar[d]\ar[r] & \sC_{d,1} \ar[d] \\
\PP(\pi^*V_d) \ar@/^/[u]^s \ar[d]\ar[r] & \PP( V_d) \ar[d]\ar[r] & \sC_{d} \\
(\PP^2)^{p+1} \ar[r]^\pi & U_d}
\end{eqnarray}

Actually the LHS column is only defined on $\pi^{-1}(U_d)$ but for simplicity we will not write this. Working in the birational category, this does not affect our results. 

We observe that the pulled-back family $\pi^* \sX_d$, seen as universal family over $\PP(\pi^*V_d)$ has a tautological rational section $s$, that is the image in $S$ of the $(p+1)^{th}$ point from $(\PP^2)^{p+1}$. This means that $\pi^* \sX_d$ has a rational point over the function field of $\PP(\pi^*V_d)$. By Kollar's Theorem $\pi^* \sX_d $ is then unirational over this field. This in turn is equivalent to the existence of a $\PP^4$-bundle on $\PP(\pi^*V_d)$ that dominates $\pi^* \sX_d$. The $\PP^4$-bundle is rational, $\pi^* \sX_d$ dominates $\sX_d$ and this in turn dominates $\sC_{d,1}$, which is then unirational.
\end{proof}

\begin{rk}
An easy adaptation of this argument allows us to show the unirationality of $\sC_{8,1}$ as well. Of course, all the planes in $\PP^5$ are projectively equivalent, and the linear system of cubics through a given 2-plane $P$ is
$|\mathcal{I}_{P/\PP^5}(3)|\cong \PP^{45}$. The plane $P$ has rational points over $\mathbb{C}$ and the universal cubic $\sX_8 \subset \PP^{45} \times \PP^5$ contains the "constant" plane $P\times \PP^{45}$. The variety $\sX_8$ also dominates $\sC_{8,1}$, by definition. Hence $\sX_8$ has rational points over the function field of $\PP^{45}$ and hence is unirational over this field. By the same argument as in Thm. \ref{teokol}, this implies the unirationality of $\sC_{8,1}$.
\end{rk}

\subsection{Special cubics in $\sC_{42}$.}\label{staglia}

The unirationality of universal cubics $\sC_{42,1}$ goes along the same lines as in the previous section, but we need to extract a couple of quite subtle results from \cite{RS1} about rational surfaces contained in cubics belonging to these divisor.

\medskip

In \cite{RS1}, the authors construct a 48-dimensional unirational family $\mathcal{S}_{42}$ of 5-nodal surfaces of degree 9 and sectional genus 2, such that the divisor $\sC_{42}$ can be described as the locus of cubics containing surfaces from this family. We reconstruct briefly their argument.

\medskip

More precisely (see \cite[Rmk. 4.3]{RS1}) they construct a family $\mathcal{W}$ of del Pezzo quintics with some special intersection theoretical properties. The construction of such surfaces depend on certain choices (5 points on $\PP^2$, a line secant to a particular rational surface, etc.) which all depend on free, rational parameters. This implies that the family $\mathcal{W}$ is unirational. By construction, each del Pezzo quintic $S\in \mathcal{W}$ is contained in one smooth del Pezzo fivefold $F\subset \PP^8$, that is an hyperplane section of the Grassmannian $Gr(2,5)$. The fivefold $F$ contains a rational, 3-dimensional, family of planes,  whose members all have class $(2,2)$ inside the Chow ring $CH^\bullet (Gr(2,5))$ of the Grassmannian. The projection with center any one of these planes defines a birational morphism between any del Pezzo quintic from $\mathcal{W}$ and some 5-nodal $S\in \mathcal{S}_{42}$. As observed in \cite{RS1}, this in turn implies that $\mathcal{S}_{42}$ is unirational. Let us denote by $\mathbb{S}$ the rational parameter space that dominates $\mathcal{S}_{42}$.
Other relevant work about $\sC_{42}$ has been done in \cite{Lai} and \cite{HS}.

\begin{rk}\label{birra}
We observe that the construction above implies that any $S\in \mathcal{S}_{42}$ is birational to a del Pezzo quintic over $\mathbb{C}(\mathbb{S})$, the function field of $\mathbb{S}$. This means that on the same field there exist a birational map $\PP^2 \dashrightarrow S$, as we explain in the following Lemma.
\end{rk}







\begin{lm}\label{univdp5}
Let $S$ be any surface in $\mathcal{S}_{42}$, then $S$ is rational over the function field of $\mathbb{S}$.
\end{lm}

\begin{proof}
As we have observed in Rmk. \ref{birra}, $S$ is birational to a del Pezzo quintic, over $\mathbb{C}(\mathbb{S})$.
By a well-known result of Enriques \cite{En}, a del Pezzo quintic is rational over any field.
\end{proof}


\begin{prop}\label{duecasi}
The universal cubic $\sC_{42,1}$ is unirational.
\end{prop}

\begin{proof}
Let us denote by $\PP(V_{42})$ the relative linear system over $\mathcal{S}_{42}$ of cubics through $S\in \mathcal{S}_{42}$, and by $\sX_{42}$ the universal family over $\PP(V_{42})$, that dominates $\sC_{42,1}$. Then the situation is the following.

\begin{eqnarray}
\xymatrix{
\rho^*\pi^* \sX_{42} \ar[r] \ar[d] & \pi^* \sX_{42} \ar[r]\ar[d] & \sX_{42} \ar[d]\ar[r] & \sC_{42,1} \ar[d] \\
\PP(\rho^*\pi^*V_{42})\ar@/^/[u]^s\ar[r]\ar[d] & \PP(\pi^*V_{42}) \ar[d]\ar[r] & \PP(V_{42}) \ar[d] \ar[r]^\varphi & \sC_{42}\\
\PP^2\times \mathbb{S} \ar[r]^\rho& \mathbb{S}\ar[r]^\pi & \mathcal{S}_{42} &}
\end{eqnarray}

Here $\pi: \mathbb{S}\to \mathcal{S}_{42}$ is the unirational parametrization described above, $\rho$ the second projection and $\varphi$ is the natural classifying map. Now, all cubics that are the fibers of the fibration $\sX_{42}\to \PP(V_{42})$ contain by construction a surface $S\in \mathcal{S}_{42}$, and since these surfaces are rational over $\mathbb{C}(\mathbb{S})$ (they are birational to del Pezzo quintics), they have rational points over the same field. Summing up, there exists a section $s:\PP(\rho^*\pi^*V_{42}) \to \rho^*\pi^*\sX_{42}$, that - thanks to Kollar's Theorem - makes $\rho^*\pi^*\sX_{42}$ unirational over $\PP^2\times \mathbb{S}$ and hence implies the unirationality of $\sC_{42,1}$.
\end{proof}


\begin{rk}
As the reader may observe, the universal cubic $\sC_{44,1}$ is missing from our description. Cubics in this divisor contain a Fano model of an Enriques surface \cite[Thm. 3.2]{Nu}. The divisor $\sC_{44}$ is unirational, so in order to apply Kollar's Theorem and show that $\sC_{44,1}$ is unirational, one should show that the generic Enriques surface has  a rational point.
\end{rk}

\subsection{Unirationality of $\sC_{d,n}$.}

In this section we are going to use the inductivity of the construction of $\sC_{d,n}$ - the moduli space of cubic fourfolds in $\sC_d$ with $n$ marked points - and Kollar's theorem in order to show the unirationality of $\sC_{d,n}$ for all $n$.

\smallskip

Recall in fact that $\sC_{d,n}$ is just the universal cubic fourfold over $\sC_{d,n-1}$. This allows us to use our machinery to prove inductively the following.

\begin{thm}
The moduli spaces $\sC_{d,n}$ are unirational for all $n$, if $8\leq d \leq 42$.
\end{thm}

\begin{proof}
We will work inductively on $n$. For $n=0$ (or respectively 1) the claim is true thanks to \cite{Nu} and \cite{RS1} (or resp. Sections \ref{sectnuer} and \ref{staglia}). We will denote by $\mathcal{S}_d$ the unirational family of rational surfaces contained in cubics in $\sC_d$, $\PP(V_d) \to \mathcal{S}_d$ the relative ideal of of cubic hypersurfaces through each surface and $\sX_{d,1}\to \PP(V_d)$ the universal cubic over the relative linear system. As seen in Thm. \ref{teokol} and Prop. \ref{duecasi}, the variety $\sX_{d,1}$ is unirational and dominates $\sC_{d,1}$ through the natural classifying map of the coarse moduli space $\sC_{d,1}$. The situation is described by the following diagram:

\begin{eqnarray}\label{diag1}
\xymatrix{
\pi^* \sX_{d,1} \ar@{->>}[r]\ar[d]\ar@/^/[rr]^{\sigma} & \sX_{d,1} \ar[d]  \ar@{->>}[r] & \sC_{d,1} \ar[d] \\
\PP(\pi^*V_d) \ar@/^/[u]^s \ar[d]\ar@{->>}[r] & \PP( V_d) \ar[d] \ar@{->>}[r] & \sC_d \\
\PP^2\times\hat{\mathcal{S}}_d\ar[r]^\pi & \mathcal{S}_d},
\end{eqnarray}

where $\PP^2\times \hat{\mathcal{S}}_{d}$ is the appropriate rational space that assures the existence of the section $s$. As seen in Thms. \ref{teokol} and \ref{duecasi}, this boils down to taking $\PP^2\times(\PP^2)^{p}$ for $8\leq d <42$, and  the rational parameter space $\PP^2\times \mathbb{S}$ for $d=42$. 

We remark that $s$ is the section that makes $\pi^* \sX_{d,1}$ unirational following Kollar's Theorem. Call $\sigma$ the classifying map to $\sC_{d,1}$. Then, we can add one point and consider the universal cubic $\sC_{d,2} \to \sC_{d,1}$, that fits in the following diagram

\begin{eqnarray}\label{diag2}
\xymatrix{
\gamma^*\sigma^*\sC_{d,2} \ar@{->>}[r]\ar[d] \ar@/^/[rr]^{\tau} & \sigma^*\sC_{d,2} \ar[d]  \ar@{->>}[r] & \sC_{d,2} \ar[d] \\
\PP^2 \times \pi^* \sX_{d,1}  \ar@/^/[u]^s\ar[d]\ar@{->>}[r]^{\gamma} & \pi^* \sX_{d,1}  \ar@{->>}[d] \ar@{->>}[r]^\sigma & \sC_{d,1} \\
\PP^2 \times \hat{\mathcal{S}}_d  \ar[r] & \hat{\mathcal{S}}_d} & ,
\end{eqnarray}

where $\gamma$ is the second projection and $\sigma$ the dominant map from Diagram \ref{diag1}. Recall that $\pi^* \sX_{d,1}$ is unirational, and dominates $\hat{\mathcal{S}}_d$.

We observe that, exactly as we have done in Diag \ref{diag1} for $\sC_{d,1}$, since $\pi^*\sX_{d,1}$ dominates $\hat{\mathcal{S}}_{d,1}$, we are able to define a natural rational section $s: \PP^{2} \times \pi^*\sX_d \to \gamma^*\sigma^* \sC_{d,2}$. 
Hence the universal cubic has rational points over $\PP^2 \times \pi^*\sX_{d,1}$ and this space is unirational. By Kollar's theorem and since  $\gamma^*\sigma^* \sC_{d,2}$ dominates $\sC_{d,2}$, we get that $\sC_{d,2}$ is unirational.

\medskip

Now it is straightforward to see how to continue the argument by induction; we just draw the diagram for the following step for clarity. Recall that $\gamma^*\sigma^*\sC_{d,2}=\tau^*\sC_{d,2}$.

\begin{eqnarray}\label{diag3}
\xymatrix{
\lambda^*\tau^*\sC_{d,3} \ar@{->>}[r]\ar[d] \ar@/^/[rr]^{\omega} & \tau^*\sC_{d,3} \ar[d]  \ar@{->>}[r] & \sC_{d,3} \ar[d] \\
\PP^2 \times \tau^*\sC_{d,2} \ar@/^/[u]^s \ar[d]\ar@{->>}[r]^\lambda & \tau^*\sC_{d,2} \ar@{->>}[d] \ar@{->>}[r]^\tau & \sC_{d,2} \\
\PP^2 \times \hat{\mathcal{S}}_d \ar[r] & \hat{\mathcal{S}}_d }
\end{eqnarray}




\end{proof}

\section{Some results of non-unirationality.}

In this last section we collect some recent results about the Kodaira dimension of moduli spaces of K3 surfaces and of Hassett divisors $\sC_d$. These allow us to show that for an infinte range of values of $d$, the universal cubic fourfold over $\sC_d$ can not be unirational. Let us first claim the following straightforward Lemma, in which we need to assume that $d$ is even.

\begin{lm}\label{easylemma}
Suppose that the Kodaira dimension of $\sF_{\frac{d+2}{2}}$ or of $\sC_d$ is positive, then the universal cubic fourfold $\sC_{d,1}$ is not unirational.
\end{lm}


We will apply Lemma \ref{easylemma} to the following two Propositions, due to Gritsenko-Hulek-Sankaran and V\'arilly-Alvarado-Tanimoto. Prop. \ref{gtk3} was initially conceived for moduli of K3 surfaces; following \cite{Nu} we write its ``translation'' in terms of cubic fourfolds via the rational map of (\ref{birat}).

\begin{prop}\label{gtk3}
Let $d > 80,\ d \equiv 2\ (mod\ 6),\ 4\not \vert  d$ be such that for any odd prime $p,\ p | d$
implies $p \equiv 1\ (mod\ 3)$. Then the Kodaira dimension of $\sC_d$ is non-negative. If moreover $d > 122$, then $\sC_d$ is of general type.
\end{prop}

\begin{prop}\label{gtCd}
The divisor $\sC_{6n+2}$ is of general type for $n > 18$ and $n \neq 20, 21, 25$ and has
nonnegative Kodaira dimension for $n>13$ and $n\neq 15$. Moreover, $\sC_{6n}$ is of general type for $n >18,\ n \neq 20, 22, 23, 25, 30, 32$ and has nonnegative Kodaira dimension for $n >16$ and $n\neq 18, 20, 22, 30.$
\end{prop}

By combining Prop. \ref{gtCd} and \ref{gtk3} with Lemma \ref{easylemma} we obtain the following Proposition.

\begin{prop}
Under the hypotheses on $d$ of Prop. \ref{gtCd} and \ref{gtk3}, the universal cubic fourfold $\sC_{d,1}$ is not unirational.
\end{prop}

The full, lenghty, statement is in Prop. \ref{notuni} in the Introduction.

\subsection{Open questions}

It would be interesting to start a systematic study of the birational geometry of universal cubic fourfolds. Some very natural questions rise, and they are of course related to the classical birational geometry of cubic fourfolds.

\medskip

\bf Questions: \rm \\

\begin{enumerate}
\item What is the Kodaira dimension of the universal cubic fourfold over $\sC_d$, for different values of $d$. 

\item In particular, is $\sC_{44,1}$ unirational?

\item Is there a relation between the Kodaira dimension of the universal cubic fourfold over $\sC_d$ and that of the universal associated K3?

\item What is the Kodaira dimension of the universal cubic fourfold over the full moduli space $\sC$?

\item Are there divisors $\sC_d$, where the geometry of special rational surfaces inside cubic fourfolds can help to describe the birational geometry of universal associated K3 surfaces?

\end{enumerate}

\bibliography{bibliography}
\bibliographystyle{amsalpha}

\end{document}